\documentclass[11pt]{article}

\usepackage[T1]{fontenc}
\usepackage{lmodern}
\usepackage{amsmath,amssymb,amsthm,mathtools}
\usepackage{graphicx}
\usepackage{placeins}
\usepackage{microtype}
\usepackage{cite}
\usepackage{xcolor}
\usepackage{enumitem}
\usepackage[a4paper,margin=30mm]{geometry}
\usepackage[hidelinks]{hyperref}
\hypersetup{pdftitle={Morse--Smale dynamics in discrete Huygens networks},pdfauthor={Jorge Buescu, Emma D'Aniello, Henrique M. Oliveira}}

\newtheorem{theorem}{Theorem}[section]
\newtheorem{proposition}[theorem]{Proposition}
\newtheorem{lemma}[theorem]{Lemma}
\newtheorem{corollary}[theorem]{Corollary}
\theoremstyle{remark}
\newtheorem{remark}[theorem]{Remark}

\newcommand{\TP}{\mathbb T_P}
\newcommand{\R}{\mathbb R}
\newcommand{\one}{\mathbf 1}
\newcommand{\Fix}{\operatorname{Fix}}
\newcommand{\Crit}{\operatorname{Crit}}
\newcommand{\grad}{\operatorname{grad}}
\newcommand{\dist}{\operatorname{dist}}

\title{Morse--Smale dynamics in discrete Huygens networks}
\author{Jorge Buescu$^{1}$, Emma D'Aniello$^{2}$, Henrique M. Oliveira$^{3,*}$}
\date{\today}

\begin{document}
\maketitle
\vspace{-1.1em}

\begin{center}
\begin{minipage}{0.90\textwidth}
\centering
\small
$^{1}$ Departamento de Matem\'atica, Faculdade de Ci\^encias, Universidade de Lisboa,\\
CEMS.UL -- Center for Mathematical Studies, Campo Grande, 1749-016 Lisbon, Portugal\par
\vspace{0.35em}
$^{2}$ Dipartimento di Matematica e Fisica, Universit\`a degli Studi della Campania ``Luigi Vanvitelli'',\\
Viale Lincoln n.~5, 81100 Caserta, Italy\par
\vspace{0.35em}
$^{3}$ Centro de An\'alise Matem\'atica, Geometria e Sistemas Din\^amicos (CAMGSD),\\
Departamento de Matem\'atica, Instituto Superior T\'ecnico, Universidade de Lisboa,\\
Av.~Rovisco Pais 1, 1049-001 Lisbon, Portugal\par
\vspace{0.35em}
$^{*}$ Corresponding author: \href{mailto:henrique.m.oliveira@tecnico.ulisboa.pt}{henrique.m.oliveira@tecnico.ulisboa.pt}
\end{minipage}
\end{center}
\vspace{0.5em}

\begin{abstract}
We study discrete networks of $N+1$ identical phase oscillators with positive
reciprocal weights and smooth odd periodic interactions, motivated by discrete
models of Huygens synchronisation. The interaction laws define an intrinsic
potential, and removing the common phase gives a gradient map on the $N$-torus.
A weighted graph Laplacian yields explicit step sizes for strict Lyapunov
decrease and global invertibility. Our main result is a Morse--Smale criterion
for maps $z\mapsto z+hK\sin z$: if $K$ is symmetric, nonsingular and weakly
diagonally dominant, the map is Morse--Smale for $0<h<1/\lambda_{\max}(K)$, and this interval is optimal.
The proof combines invariant coordinate slabs with a maximum principle for the
adjoint variational equation. The criterion applies to unit-weight paths of any
length and to the three-leaf star, and a weighted version covers paths with
discretely concave weights.
By structural stability, it extends to every graph containing a spanning path
with sufficiently weak additional links.
Conversely, we construct infinitely many trees with Morse potentials and an
almost-global sink, whose maps nevertheless have non-transverse heteroclinic
intersections.	
	\end{abstract}

\vspace{0.4em}
\begin{center}
\begin{minipage}{0.90\textwidth}
\end{minipage}
\end{center}
\vspace{0.5em}
\clearpage

\section{Introduction}
\label{sec:introduction}
Networks of coupled phase oscillators describe synchronisation and phase
locking in physical, biological and engineered systems
\cite{Kuramoto1984,PikovskyRosenblumKurths2001,Strogatz2000,AcebronEtAl2005,DorflerBullo2014},
and pairwise potentials and graph Laplacians relate interaction laws to
continuous-time dynamics \cite{ScardoviSarletteSepulchre2007,DorflerBullo2014}.
For discrete maps the update size also matters: a potential decreasing along
a gradient flow may increase after a finite step, and the map may lose
invertibility or acquire non-trivial periodic orbits.  Discrete phase models
appear in coordination algorithms on tori
\cite{SarletteSepulchreLeonard2006,KleinLeeMorgansenJavidi2008,KleinMorgansen2008}
and in studies of discrete Kuramoto models, deterministic and stochastic
\cite{MaistrenkoEtAl2008,ZhangZhu2020DiscreteKuramoto,Shim2020DiscreteKuramoto,
	KalmarNagyHorvathVincze2025,Wu2026DiscreteKuramoto,JafarianMamduhiJohansson2024}.

Our motivation is discrete Huygens-type models, in which one update is the
net effect of mutual impacts over one cycle, so the interaction strength $h$
is a physical parameter.  After the two-oscillator model
\cite{OliveiraMelo2015}, three-oscillator models with all-to-all
\cite{DAnielloOliveira2024} and aligned nearest-neighbour
\cite{BuescuDAnielloOliveira2025Aligned} interaction were studied, with an
explicit Lyapunov function in \cite{BuescuDAnielloOliveira2025Lyapunov} and
Morse--Smale and structural stability properties in
\cite{BuescuOliveira2026MorseSmale}.  Here we allow arbitrary connected
undirected weighted graphs, with reciprocal coupling (each pairwise
interaction acts equally and oppositely on its two oscillators) and a
possibly different smooth periodic odd law on each edge.  For $g=\sin$ the
coupling is repulsive in the Kuramoto convention, so phase opposition, the
anti-phase state observed by Huygens, is the attractor.

The interaction laws define an intrinsic discrete potential, and removing
the common phase gives an $N$-dimensional gradient map with an explicit
constant metric.  Standard gradient and spectral arguments, with a weighted
graph Laplacian supplying the bounds, give explicit intervals of $h$ for
strict Lyapunov decrease, global invertibility and almost-everywhere
convergence to minima.  On sinusoidal trees the $2^N$ critical
configurations are explicit, and phase opposition is the unique sink.

Our main result is a matrix criterion for Morse--Smale dynamics.
For a vector $z=(z_1,\ldots,z_N)^T$, we write
$\sin z=(\sin z_1,\ldots,\sin z_N)^T$.
We prove that if $K$ is symmetric, nonsingular and weakly diagonally
dominant, then $z\mapsto z+hK\sin z$ is Morse--Smale on $\mathbb{T}^N$
for every $0<h<1/\lambda_{\max}(K)$; such a matrix is automatically
positive definite.
 This interval is optimal: at $h=1/\lambda_{\max}(K)$ the map ceases to be
 a diffeomorphism.
The proof combines two consequences of diagonal dominance: each coordinate
slab $\{\cos z_i\le0\}$ is forward invariant, and the maximum of a solution
of the adjoint variational equation propagates forward inside such a slab.
A non-transverse intersection would produce an adjoint solution decaying at
both ends, which this propagation rules out.
General discretisation results \cite{Garay1996Discretization} give
Morse--Smale dynamics for the map only when the gradient flow is already
known to be Morse--Smale, and only for sufficiently small, non-explicit
$h$.  Our criterion needs no information about the flow and holds on an explicit 
and optimal 
interval of $h$, which matters here because $h$ is a physical
interaction strength rather than a step size that may be taken small. 
The criterion applies to unit-weight paths of every finite length and to
the star with three leaves.
It extends to weighted paths whose weights are discretely concave.
By structural stability, every graph containing a spanning path admits an
open set of positive weights with Morse--Smale dynamics, when the
additional links are sufficiently weak.  A phase translation transfers all
of this to the opposite sinusoidal sign.

Without diagonal dominance, transversality may fail.  
 We construct an infinite family of trees,
with $4r+2$ vertices for every $r\ge1$, whose maps have Morse potentials,
strict Lyapunov decrease and a unique sink attracting almost every orbit,
yet have non-transverse heteroclinic intersections for every admissible
$h$, however small.  The mechanism is an invariant circle joining two fixed
points, along which the stable and unstable tangent spaces coincide in
dimension $3r+1$.  The failure is not forced by dimension: the index
difference $2r+1$ would allow a transverse intersection.  We give the
six-vertex case explicitly.  Hence convergence and hyperbolicity alone do
not imply the Morse--Smale property.
 Since these maps are not diffeomorphisms for larger $h$, they are never Morse--Smale.
 
Sections~\ref{sec:reciprocal-gradient} and~\ref{sec:finite-step}
establish the potential and the finite-step estimates.
Section~\ref{sec:morse-smale} proves the matrix criterion, the optimality
of its step-size interval and a weighted version.
Section~\ref{sec:graph-topology} applies the criterion to networks, derives
the structural stability consequences, constructs the tree counterexamples
and describes the constraints imposed by cycles.
Section~\ref{sec:applications} discusses physical connections, and
Section~\ref{sec:conclusion} concludes.

\section{The intrinsic discrete potential and quotient dynamics}
\label{sec:reciprocal-gradient}

We consider a network of $N+1$ identical oscillators, $N\ge1$, each
moving along a cycle of period $P>0$.  The state of oscillator $i$ is its
phase $\theta_i\in\TP=\R/P\mathbb Z$, and the state of the network is
$\theta=(\theta_1,\ldots,\theta_{N+1})^T\in\TP^{N+1}$.
The oscillators interact in pairs.  We record which pairs interact by a
connected undirected graph $\mathcal G=(\mathsf V,\mathcal E)$ with
vertex set $\mathsf V=\{1,\ldots,N+1\}$: oscillators $i$ and $j$ interact
if and only if $\{i,j\}\in\mathcal E$.

Each interacting pair has its own coupling strength and interaction law.
To each edge $e$ we assign a weight $a_e>0$ and a smooth function
$g_e:\R\to\R$, which describes how the phase correction produced by the
interaction depends on the phase difference.  We assume
\begin{equation}
	\label{eq:odd-periodic}
	g_e(s+P)=g_e(s),\qquad g_e(-s)=-g_e(s).
\end{equation}
Periodicity reflects the fact that phases are defined modulo $P$.
Since both oscillators of an interacting pair share the weight $a_e$ and
the law $g_e$, oddness makes the coupling reciprocal: the two oscillators
receive equal and opposite corrections.

We work in a co-rotating frame, in which the common phase advance of the
uncoupled oscillators is removed.  One update of the network then adds to
each phase the corrections produced by all its neighbours, scaled by the
interaction strength $h>0$:
\begin{equation}
	\label{eq:network-components}
	\theta_i\longmapsto\theta_i
	+h\sum_{j:\{i,j\}\in\mathcal E}a_{\{i,j\}}\,g_{\{i,j\}}(\theta_i-\theta_j)
	\pmod P,\qquad i=1,\ldots,N+1 .
\end{equation}
Since each $g_e$ is odd, the right-hand side does not depend on which endpoint of an
edge is written first.

To write \eqref{eq:network-components} in matrix form, choose an
orientation and an ordering of the edges.  For an edge $e$ with head $i$
and tail $j$, put $b_e=\mathrm e_i-\mathrm e_j$, where $\mathrm e_i$ are
the standard coordinate vectors.  The incidence matrix
$B=(b_e)_{e\in\mathcal E}\in\R^{(N+1)\times|\mathcal E|}$ satisfies
$B^T\one=0$, where $\one=(1,\ldots,1)^T\in\R^{N+1}$, and $B^T\theta$ is
the vector of phase differences across the edges.  With
$A=\operatorname{diag}(a_e)$ and
$\mathbf g(z)=(g_e(z_e))_{e\in\mathcal E}$, the network map is
\begin{equation}
	\label{eq:network}
	F_h(\theta)=\theta+hBA\mathbf g(B^T\theta)\pmod P,
	\qquad h>0.
\end{equation}
Write $\widetilde F_h$ for the lift to $\R^{N+1}$ given by the same
formula without reduction modulo $P$.  We use the same symbols for scalar
potentials and their periodic lifts.  Vector identities on tori are
understood modulo $P$; differentiation uses real representatives.

%

Oddness also makes the network map a gradient step.  For each edge, let
\begin{equation}
	\label{eq:Phi-def}
	\Phi_e(s)=-\int_0^s g_e(u)\,du .
\end{equation}
Since $g_e$ is odd, $\Phi_e$ is even; since a periodic odd function has
zero integral over a period, $\Phi_e$ is $P$-periodic.  Thus $\Phi_e$ is
a smooth, even, $P$-periodic function with $\Phi_e'=-g_e$, and every
other primitive of $-g_e$ differs from it by a constant.  We define the
edge potential $\Psi$ and the intrinsic discrete potential
$\mathcal V$ by
\begin{equation}
	\label{eq:full-potential}
	\Psi(z)=\sum_{e\in\mathcal E}a_e\Phi_e(z_e),
	\quad z\in\TP^{|\mathcal E|},
	\qquad \mathcal V(\theta)=\Psi(B^T\theta).
\end{equation}
For $g_e(s)=\sin s$ one may take $\Phi_e(s)=\cos s$, so
$\mathcal V(\theta)=\sum_e a_e\cos(b_e^T\theta)$, up to a constant. Pairwise phase potentials are classical
\cite{Kuramoto1984,ScardoviSarletteSepulchre2007,DorflerBullo2014};
discrete descent constructions on phase tori also have antecedents
\cite{SarletteSepulchreLeonard2006,KleinMorgansen2008}.

\begin{proposition}[Gradient structure and phase symmetry]
	\label{lem:edge-potential}
	\label{thm:full-gradient}
	\label{thm:rotation}
	The potential $\mathcal V$ is independent of the edge orientation and
	determined up to an additive constant.  Moreover:
	\begin{enumerate}[label=\textup{(\roman*)}]
		\item The network map is a gradient step for $\mathcal V$:
		\begin{equation}
			\label{eq:full-gradient-components}
			\nabla\Psi(z)=-A\mathbf g(z),\qquad
			\nabla\mathcal V(\theta)=-BA\mathbf g(B^T\theta),
		\end{equation}
		\begin{equation}
			\label{eq:full-gradient-step}
			\widetilde F_h(\theta)=\theta-h\nabla\mathcal V(\theta).
		\end{equation}
		\item For every $\alpha\in\R$,
		\begin{equation}
			\label{eq:F-rotation}
			\mathcal V(\theta+\alpha\one)=\mathcal V(\theta),\qquad
			\widetilde F_h(\theta+\alpha\one)=\widetilde F_h(\theta)+\alpha\one .
		\end{equation}
		\item $\one^T\nabla\mathcal V(\theta)=0$, and the lifted map preserves
		the phase sum:
		\begin{equation}
			\label{eq:sum-conserved}
			\one^T\widetilde F_h(\theta)=\one^T\theta .
		\end{equation}
	\end{enumerate}
	In particular, $F_h$ descends to the $N$-dimensional quotient
	$\TP^{N+1}/\TP\simeq\TP^N$ by common-phase shifts.
\end{proposition}

\begin{proof}
	Differentiation and the chain rule give \eqref{eq:full-gradient-components}
	and \eqref{eq:full-gradient-step}.  Reversing an edge changes the signs of
	both $b_e$ and $b_e^T\theta$; evenness of $\Phi_e$ preserves the
	potential, and oddness of $g_e$ preserves the map.  The only freedom in
	the construction is the additive constant in each primitive.  Finally,
	$B^T\one=0$ gives \eqref{eq:F-rotation}, and $\one^TB=0$ gives the
	orthogonality and \eqref{eq:sum-conserved}.
\end{proof}

To write the quotient map explicitly, we use the last phase as a
reference.  Let $I_N$ be the identity matrix and $\one_N\in\R^N$ the
all-ones vector, and set
\[
	Q=\begin{pmatrix}I_N&-\one_N\end{pmatrix},
	\qquad S=\begin{pmatrix}I_N\\0\end{pmatrix},
\]
so that $Q\in\R^{N\times(N+1)}$ and $S\in\R^{(N+1)\times N}$, whose last
row is zero.  Then $QS=I_N$, $Q\one=0$, and
\[
	x=Q\theta
	=(\theta_1-\theta_{N+1},\ldots,\theta_N-\theta_{N+1})^T\in\TP^N.
\]
The section $Sx$ sets the last phase to zero.  Since
$\theta=Sx+\theta_{N+1}\one$, the edge differences satisfy
$z=B^T\theta=Cx\in\TP^{|\mathcal E|}$, where
\begin{equation}
	\label{eq:reduced-matrices}
	C=B^TS,\qquad M=QQ^T=I_N+\one_N\one_N^T,\qquad
	M^{-1}=I_N-\frac{\one_N\one_N^T}{N+1}.
\end{equation}
The eigenvalues of $M$ are $1$, with multiplicity $N-1$, and $N+1$, with
multiplicity one.  Equip the quotient with the constant metric
\begin{equation}
	\label{eq:metric}
	\langle u,v\rangle_{M^{-1}}=u^TM^{-1}v,\qquad u,v\in\R^N,
\end{equation}
and denote its Riemannian gradient by $\grad_{M^{-1}}$.

%
%
%
%

\begin{proposition}[Reduced potential and gradient map]
	\label{thm:inherited-potential}
	\label{thm:reduced-gradient}
	There is a unique smooth \(V:\TP^N\to\R\) satisfying
	\(\mathcal V=V\circ Q\). It is given by
	\begin{equation}
		\label{eq:reduced-potential}
		V(x)=\Psi(Cx),\qquad \nabla V(x)=-C^TA\mathbf g(Cx).
	\end{equation}
	The induced map \(\bar F_h:\TP^N\to\TP^N\) is
	\begin{equation}
		\label{eq:reduced-gradient-map}
		\bar F_h(x)=x-hM\nabla V(x)
		=x-h\grad_{M^{-1}}V(x),
		\qquad \grad_{M^{-1}}V=M\nabla V.
	\end{equation}
\end{proposition}

\begin{proof}
	By \eqref{eq:F-rotation}, \(\mathcal V\) is constant on each fibre of
	\(Q\); evaluation at \(Sx\) gives the unique factor \(V(x)=\Psi(Cx)\).
	Differentiation gives its gradient and
	\begin{equation}
		\label{eq:chain-gradient}
		\nabla_\theta\mathcal V(\theta)=Q^T\nabla V(Q\theta).
	\end{equation}
	Apply \(Q\) to \eqref{eq:full-gradient-step} and use \(QQ^T=M\).
	The metric \eqref{eq:metric} then gives the stated Riemannian gradient.
\end{proof}

The choice of the last phase as reference is immaterial: any other
reference phase gives a linearly conjugate representation.  The
associated continuous gradient equations, with continuous time \(\tau\),
are
\begin{equation}
	\label{eq:full-flow}
	\frac{d\theta}{d\tau}=-\nabla\mathcal V(\theta),
	\qquad \frac{dx}{d\tau}=-M\nabla V(x).
\end{equation}
Along a non-equilibrium reduced trajectory,
\(dV/d\tau=-\nabla V^TM\nabla V<0\).
The discrete map has the algebraic form of an explicit Euler step of
\eqref{eq:full-flow}.  Since \(h\) is a physical interaction strength
rather than a step size that may be taken small, descent after a finite
update is not automatic and requires an explicit bound on \(h\).

%
%
%
%
%
%
%
%
%
%
%
%
%
%
%
%
%
%

\section{Finite-step global dynamics}
\label{sec:finite-step}

   We now bound the admissible interaction strength $h$ by the spectrum of a
weighted graph Laplacian. 
Use the Euclidean norm for vectors and its induced norm \(\|\cdot\|_2\)
for matrices. Define
\begin{equation}
\label{eq:Lbeta}
 \beta_e=a_e\|g_e'\|_\infty,\qquad
 W_\beta=\operatorname{diag}(\beta_e),\qquad
 L_\beta=BW_\beta B^T,
\end{equation}
where \(\|g_e'\|_\infty=\sup_{s\in\R}|g_e'(s)|\), and put
\begin{equation}
\label{eq:Lambda}
 \Lambda=\lambda_{\max}(L_\beta).
\end{equation}
Here \(\lambda_{\max}\) denotes the largest eigenvalue of a symmetric
matrix. If \(\Lambda=0\), all interaction laws vanish by oddness and
\(F_h\) is the identity. Henceforth assume \(\Lambda>0\).
The matrix \(L_\beta\) is a weighted graph Laplacian
\cite{Chung1997,DorflerBullo2014}.

\begin{lemma}[Hessian bound]
	\label{thm:Hessian-bound}
	For every \(\theta\in\TP^{N+1}\) and \(v\in\R^{N+1}\),
	\begin{equation}
		\label{eq:Hessian-quadratic}
		|v^TD^2\mathcal V(\theta)v|
		\le v^TL_\beta v\le\Lambda\|v\|^2,
		\qquad \|D^2\mathcal V(\theta)\|_2\le\Lambda.
	\end{equation}
\end{lemma}

\begin{proof}
	With \(\Gamma(\theta)=\operatorname{diag}
	(a_eg_e'((B^T\theta)_e))\), differentiation gives
	\begin{equation}
		\label{eq:Hessian-edge-form}
		D^2\mathcal V(\theta)=-B\Gamma(\theta)B^T.
	\end{equation}
	Thus
	\[
	|v^TD^2\mathcal V(\theta)v|
	\le\sum_e\beta_e(b_e^Tv)^2=v^TL_\beta v
	\le\Lambda\|v\|^2,
	\]
	and symmetry gives the norm bound.
\end{proof}

\begin{theorem}[Descent and invertibility]
	\label{thm:finite-estimates}
	\label{thm:full-Lyapunov}
	\label{thm:reduced-Lyapunov}
	\label{thm:diffeomorphism}
	For \(0<h<2/\Lambda\), let \(c_h=h(1-h\Lambda/2)>0\). Then
	\begin{equation}
		\label{eq:reduced-descent}
		\begin{aligned}
			\mathcal V(F_h(\theta))
			&\le\mathcal V(\theta)-c_h\|\nabla\mathcal V(\theta)\|^2,\\
			V(\bar F_h(x))
			&\le V(x)-c_h\nabla V(x)^TM\nabla V(x),
		\end{aligned}
	\end{equation}
	so both potentials decrease strictly away from their critical sets.
	If \(0<h<1/\Lambda\), then \(F_h\) and \(\bar F_h\) are smooth global
	diffeomorphisms.
\end{theorem}

\begin{proof}
	Taylor's inequality on the periodic lift, with increment
	\(-h\nabla\mathcal V(\theta)\), and Lemma~\ref{thm:Hessian-bound} give
	the first descent estimate \cite{Nesterov2004}. The identities
	\(\mathcal V=V\circ Q\) and \eqref{eq:chain-gradient} give
	\[
	\|\nabla\mathcal V(\theta)\|^2
	=\nabla V(Q\theta)^TM\nabla V(Q\theta),
	\]
	and hence the reduced estimate. Positive definiteness of \(M\) gives
	strictness.
	
	For \(h<1/\Lambda\), Lemma~\ref{thm:Hessian-bound} shows that every eigenvalue of
	\(D\widetilde F_h=I_{N+1}-hD^2\mathcal V\) is positive. The induced
	torus map is therefore an orientation-preserving local diffeomorphism.
	Its image is open and compact, hence closed, so connectedness makes
	it surjective. The homotopy
	\(\theta\mapsto\theta-th\nabla\mathcal V(\theta)\), \(0\le t\le1\),
	gives degree one. A local diffeomorphism of a compact manifold is a
	covering map, and positive orientation makes this covering one-sheeted
	\cite{Lee2012SmoothManifolds}. Thus \(F_h\) is a global diffeomorphism.
	Its inverse also commutes with common-phase shifts, so it induces
	a diffeomorphism on the quotient.
\end{proof}

\begin{remark}[Common interaction law]
	\label{cor:common-law}
	If \(g_e=g\) for every edge, then \(W_\beta=\|g'\|_\infty A\), so with
	\(L_a=BAB^T\),
	\[
		\Lambda=\|g'\|_\infty\lambda_{\max}(L_a).
	\]
	In particular, \(\Lambda=(2\pi/P)\lambda_{\max}(L_a)\) for
	\(g(s)=\sin(2\pi s/P)\), and
	\(\Lambda=\lambda_{\max}(L_a)\) for the standard law \(g(s)=\sin s\).
\end{remark}

The full and reduced descent estimates use the same parameter bound:
the metric introduces no further restriction. The interval
\(0<h<1/\Lambda\) ensures both descent and invertibility, whereas
\(0<h<2/\Lambda\) ensures descent without requiring invertibility.
These bounds are optimal for sinusoidal trees
(Corollary~\ref{cor:tree-sharp}), but can be improved for other graphs or
interaction laws.

Write \(\Crit(W)=\{x:\nabla W(x)=0\}\) and \(\Fix(F)=\{x:F(x)=x\}\).
The distance \(\dist\) and measure \(\operatorname{Leb}\) are induced
by the flat coordinates on \(\TP^N\). A critical point is
non-degenerate when its Hessian is nonsingular; its Morse index
\(\operatorname{ind}_{\mathrm M}\) counts the negative Hessian
eigenvalues. A function is Morse if all its critical points are
non-degenerate \cite{Milnor1963}. At a hyperbolic fixed point, write
\(W^s(p)\) and \(W^u(p)\) for the stable and unstable manifolds.
For any fixed point, its basin is
\[
 \mathcal B(p)=\{x\in\TP^N:\bar F_h^{\,n}(x)\to p\}.
\]

\begin{proposition}[Fixed points and convergence]
	\label{thm:fixed-critical}
	\label{cor:no-periodic}
	\label{thm:approach-critical}
	\label{thm:finite-convergence}
	Let \(0<h<2/\Lambda\). Then
	\begin{equation}
		\label{eq:fixed-critical}
		\Fix(\bar F_h)=\Crit(V),
	\end{equation}
	and every periodic orbit of \(\bar F_h\) is fixed. Every forward orbit
	\(x_{n+1}=\bar F_h(x_n)\) satisfies
	\begin{equation}
		\label{eq:distance-critical}
		\nabla V(x_n)\to0,\qquad
		\dist(x_n,\Crit(V))\to0;
	\end{equation}
	in particular, every accumulation point is fixed. If \(\Crit(V)\) is
	finite, each orbit converges to one of its points.
\end{proposition}

\begin{proof}
	Critical points are fixed by \eqref{eq:reduced-gradient-map}.
	Conversely, a fixed point leaves \(V\) unchanged, so strict descent
	forces criticality. A periodic orbit cannot contain a strict decrease
	and is therefore fixed.
	
	Summing \eqref{eq:reduced-descent} gives
	\[
	c_h\sum_{n=0}^{m-1}\nabla V(x_n)^TM\nabla V(x_n)
	\le V(x_0)-V(x_m).
	\]
	Compactness bounds the right-hand side, so positive definiteness of
	\(M\) gives \(\nabla V(x_n)\to0\). Every accumulation point is
	therefore critical. If the distance to \(\Crit(V)\) did not tend to
	zero, a convergent subsequence at a positive distance from it would
	contradict this conclusion. Also,
	\[
		\dist(x_{n+1},x_n)\le h\|M\nabla V(x_n)\|\longrightarrow0.
	\]
	When the critical set is finite, choose balls around its points with
	pairwise disjoint closures. Their mutual distances have a positive
	lower bound. Eventually the orbit lies in their union and its steps
	are smaller than this bound, so it remains in one ball. Its distance
	to the critical set tends to zero, and hence it converges to that
	ball's critical point.
\end{proof}

\begin{proposition}[Hyperbolicity and Morse index]
	\label{thm:Morse-hyperbolic}
	Let \(0<h<1/\Lambda\). Each non-degenerate critical point \(p\) of \(V\)
	is a hyperbolic fixed point of \(\bar F_h\), and
	\begin{equation}
		\label{eq:index-equality}
		\dim W^u(p)=\operatorname{ind}_{\mathrm M}(p).
	\end{equation}
	Such a point is a sink exactly when it is a strict local minimum,
	a source exactly when it is a strict local maximum, and a saddle
	otherwise.
\end{proposition}

\begin{proof}
	At a non-degenerate critical point, put \(H_p=D^2V(p)\). Then
	\[
		D\bar F_h(p)=I_N-hMH_p.
	\]
	The matrix \(MH_p\) is similar to \(M^{1/2}H_pM^{1/2}\), which is
	symmetric and has the inertia of \(H_p\)
	\cite{HornJohnson2012}. Its eigenvalues \(\mu\) are thus real and
	nonzero. They are also the nonzero eigenvalues of
	\(D^2\mathcal V(\theta)=Q^TH_pQ\), for \(Q\theta=p\), so
	\(|\mu|\le\Lambda\) by Lemma~\ref{thm:Hessian-bound}. For \(0<h<1/\Lambda\), positive \(\mu\)
	give multipliers \(0<1-h\mu<1\), while negative \(\mu\) give
	\(1-h\mu>1\). This proves hyperbolicity, the index formula and the
	classification.
\end{proof}

\begin{remark}
	The same multiplier calculation gives \(-1<1-h\mu<1\) for positive
	\(\mu\), and \(1-h\mu>1\) for negative \(\mu\), throughout
	\(0<h<2/\Lambda\). Thus non-degenerate critical points retain
	multipliers off the unit circle in the full descent interval.
\end{remark}

\begin{theorem}[Almost-global convergence]
	\label{thm:global-dynamics}
	\label{cor:almost-everywhere}
	Let \(0<h<1/\Lambda\) and suppose that \(V\) is Morse. Then every orbit
	of \(\bar F_h\) converges to a critical point, and almost every orbit
	converges to a local minimum. If there is a unique local minimum
	\(p_*\), then
	\begin{equation}
		\label{eq:full-measure-basin}
		\operatorname{Leb}\bigl(\TP^N\setminus\mathcal B(p_*)\bigr)=0.
	\end{equation}
\end{theorem}

\begin{proof}
	A Morse function on the compact torus has finitely many critical points,
	so every orbit converges by Proposition~\ref{thm:finite-convergence}.
	At each non-minimum critical point, the local stable manifold has
	dimension less than \(N\) by Proposition~\ref{thm:Morse-hyperbolic} and
	the stable-manifold theorem \cite{HirschPughShub1977}.
	It has zero Lebesgue measure, and global invertibility gives
	\[
	W^s(p)=\bigcup_{k\ge0}\bar F_h^{-k}
	\bigl(W^s_{\mathrm{loc}}(p)\bigr).
	\]
	Every preimage is a smooth submanifold of the same dimension and
	therefore has measure zero \cite{Lee2012SmoothManifolds}.
	The finite union over non-minimum critical points is null.
	All other initial conditions converge to local minima.
\end{proof}

Theorem~\ref{thm:global-dynamics} uses the smaller interval
\(0<h<1/\Lambda\) to ensure global invertibility. Its stable-manifold
argument is standard in gradient iteration
\cite{LeeSimchowitzJordanRecht2016,SchaefferMcCalla2020}; here the
graph Laplacian gives the required bound on \(h\).

\section{A matrix criterion for Morse--Smale dynamics}
\label{sec:morse-smale}

The potential estimates do not settle whether stable and unstable
manifolds intersect transversely.  We now prove a sufficient matrix
condition for this property.  The network consequences and the tree
counterexamples follow in Section~\ref{sec:graph-topology}.

Write $\mathbb T=\R/(2\pi\mathbb Z)$.  For vectors, sine and cosine act
componentwise:
\[
\sin z=(\sin z_1,\ldots,\sin z_N)^T,\qquad
\cos z=(\cos z_1,\ldots,\cos z_N)^T.
\]
For a symmetric matrix $K=(k_{ij})\in\R^{N\times N}$ and $h>0$, consider
\begin{equation}
	\label{eq:matrix-ms-map}
	G_{K,h}(z)=z+hK\sin z\pmod{2\pi},
	\qquad z\in\mathbb T^N.
\end{equation}
When $K$ is positive definite, $G_{K,h}$ is a gradient step of
$V_0(z)=\sum_{i=1}^N\cos z_i$ for the constant metric $K^{-1}$.  The
matrix need not arise from a graph.

A point is \emph{nonwandering} if every neighbourhood of it meets some
positive iterate of itself.  A diffeomorphism is \emph{Morse--Smale} if its
nonwandering set consists of finitely many hyperbolic periodic points and
all their stable and unstable manifolds intersect transversely: for all
such points $p,q$ and every $z\in W^u(p)\cap W^s(q)$,
\begin{equation}
	\label{eq:path-transversality}
	T_zW^u(p)+T_zW^s(q)=T_z\mathbb T^N.
\end{equation}
We say that $K$ is \emph{weakly diagonally dominant} if
\begin{equation}
	\label{eq:matrix-dominance}
	k_{ii}\ge\sum_{j\ne i}|k_{ij}|,
	\qquad i=1,\ldots,N.
\end{equation}
By Gershgorin's theorem, a symmetric matrix satisfying
\eqref{eq:matrix-dominance} is positive semidefinite; it is therefore
positive definite if and only if it is nonsingular.

\begin{theorem}[A matrix criterion for Morse--Smale dynamics]
	\label{thm:matrix-morse-smale}
	Let $K$ be symmetric, nonsingular and weakly diagonally dominant, and let
	\begin{equation}
		\label{eq:matrix-ms-range}
		0<h<\frac{1}{\lambda_{\max}(K)}.
	\end{equation}
	Then $G_{K,h}$ is a Morse--Smale diffeomorphism of $\mathbb T^N$ with
	nonwandering set $\{0,\pi\}^N$.  It has $\binom Nr$ fixed points of
	unstable dimension $r$, for $r=0,\ldots,N$, and a unique sink
	$(\pi,\ldots,\pi)$, whose basin has full Lebesgue measure. The interval
	\eqref{eq:matrix-ms-range} is optimal: for $h\ge1/\lambda_{\max}(K)$,
	$G_{K,h}$ is not a local diffeomorphism.
\end{theorem}

The proof uses invariant coordinate slabs and the adjoint variational
equation.  We first establish these two ingredients.

\begin{lemma}[Invariant coordinate slabs]
	\label{lem:path-slabs}
	Under the hypotheses of Theorem~\ref{thm:matrix-morse-smale}, each set
	\[
	\mathcal S_i=\{z\in\mathbb T^N:
	z_i\in[\pi/2,3\pi/2]\pmod{2\pi}\},
	\qquad i=1,\ldots,N,
	\]
	is forward invariant under $G_{K,h}$.  Equivalently, along any orbit
	$z_n=G_{K,h}^n(z_0)$, with components $z_{i,n}$,
	$\cos z_{i,n}\le0$ implies $\cos z_{i,n+1}\le0$.
\end{lemma}

\begin{proof}
	Fix $i$ and let $s=z_{i,n}\in[\pi/2,3\pi/2]$.  The $i$th component of the
	update is
	\[
	s\longmapsto f(s)+hb,\qquad f(s)=s+hk_{ii}\sin s,\qquad
	b=\sum_{j\ne i}k_{ij}\sin z_{j,n},
	\]
	and $|b|\le\sum_{j\ne i}|k_{ij}|\le k_{ii}$ by
	\eqref{eq:matrix-dominance}.  Since
	$hk_{ii}\le h\lambda_{\max}(K)<1$, the function $f$ is strictly
	increasing and maps $[\pi/2,3\pi/2]$ onto
	$[\pi/2+hk_{ii},\,3\pi/2-hk_{ii}]$.  Hence
	$f(s)+hb\in[\pi/2,3\pi/2]$.
\end{proof}

We represent covectors, that is, linear functionals on tangent vectors,
by column vectors in the standard coordinates of the torus.  The
following criterion is classical \cite{Palmer1988}; we include a proof
for completeness.

\begin{lemma}[Adjoint criterion for transversality]
	\label{lem:adjoint-transversality}
	Let $G$ be a $C^1$ diffeomorphism of $\mathbb T^N$, let $p,q$ be hyperbolic
	fixed points, and let $z_n=G^n(z_0)$ with
	$z_0\in W^u(p)\cap W^s(q)$.  Set $A_n=DG(z_n)$.
	The intersection at $z_0$ is transverse if and only if the only solution of
	\begin{equation}
		\label{eq:adjoint-criterion}
		\eta_n=A_n^T\eta_{n+1},\qquad n\in\mathbb Z,
		\qquad
		\eta_n\longrightarrow0\quad\text{as }n\to\pm\infty,
	\end{equation}
	is $\eta_n=0$ for all $n$.
\end{lemma}

\begin{proof}
	Write $E^s_n=T_{z_n}W^s(q)$ and $E^u_n=T_{z_n}W^u(p)$.
	The adjoint equation transports covectors by
	$\eta_{n+1}=A_n^{-T}\eta_n$ and preserves their pairing with solutions of
	$v_{n+1}=A_nv_n$.
	
	We claim that $\eta_n\to0$ as $n\to+\infty$ precisely when $\eta_0$
	vanishes on $E^s_0$.  Choose $C^1$ local coordinates $(s,u)$ near $q$ in
	which its local stable manifold is $\{u=0\}$, using the stable-manifold
	theorem \cite{HirschPughShub1977}.  For all sufficiently large $n$, the
	derivative in these coordinates has the block form
	\[
	\widehat A_n=\begin{pmatrix}\Sigma_n&\Xi_n\\0&\Upsilon_n\end{pmatrix},
	\]
	where $\Sigma_n$ tends to the stable block and $\Upsilon_n$ to the unstable
	block of $DG(q)$.  If $\eta_0$ vanishes on $E^s_0$, invariance of the stable
	tangent spaces makes the transported covector have the form $(0,\beta_n)^T$,
	with $\beta_{n+1}=\Upsilon_n^{-T}\beta_n$.
	The limiting inverse transpose has all eigenvalues inside the unit circle.
	Choose a norm in which its operator norm is less than one.  A uniform
	bound less than one then holds for all sufficiently large $n$.
	The coordinate changes and their inverses have bounded derivatives along
	the convergent orbit, so $\eta_n$ decays exponentially in the original
	coordinates too.
	Conversely, stable tangent vectors evolve through $\Sigma_n$; the same norm
	argument gives their exponential decay.  If $\eta_n\to0$, the constant
	pairing $\eta_n^Tv_n=\eta_0^Tv_0$ must therefore vanish, so $\eta_0$
	vanishes on $E^s_0$.
	
	Applying the same argument to $G^{-1}$ at $p$ identifies backward decay
	with vanishing on $E^u_0$.  Solutions of \eqref{eq:adjoint-criterion}
	therefore correspond exactly to covectors vanishing on $E^s_0 + E^u_0$.
	A nonzero such covector exists precisely when the intersection is not
	transverse.
\end{proof}

\begin{proof}[Proof of Theorem~\ref{thm:matrix-morse-smale}]
	By the observation preceding the theorem, $K$ is positive definite.  Put
	$\lambda=\lambda_{\max}(K)$ and
	$D(z)=\operatorname{diag}(\cos z_1,\ldots,\cos z_N)$.
	The derivative $I_N+hKD(z)$ is similar to the symmetric matrix
	$I_N+hK^{1/2}D(z)K^{1/2}$, whose eigenvalues lie in
	$[1-h\lambda,1+h\lambda]\subset(0,2)$.  Thus $G_{K,h}$ is a local
	diffeomorphism.  It is homotopic to the identity through
	$z\mapsto z+thK\sin z$, $0\le t\le1$, and hence has degree one.
	Compactness and positivity of the Jacobian determinant make it a
	one-sheeted covering, so it is a global diffeomorphism. 
	
	The interval \eqref{eq:matrix-ms-range} cannot be enlarged.  Consider the
	diagonal points $z(t)=t\one_N$, $t\in[0,\pi]$, where
	$D(z(t))=\cos t\,I_N$ and hence $DG_{K,h}(z(t))=I_N+h\cos t\,K$.  If
	$Kv=\lambda v$ with $v\neq0$, then $DG_{K,h}(z(t))\,v=(1+h\lambda\cos t)\,v$
	for every $t$.  If $h\lambda\ge1$, the number
	$t_*=\arccos\bigl(-1/(h\lambda)\bigr)$ lies in $(\pi/2,\pi]$, and
	$DG_{K,h}(z(t_*))\,v=0$, so $G_{K,h}$ is not a local diffeomorphism.
	
	Since $\|D^2V_0\|_2\le1$ and $\|Kv\|^2\le\lambda\,v^TKv$ for
	$v\in\R^N$, Taylor's inequality on a real lift gives
	\begin{equation}
		\label{eq:matrix-ms-descent}
		V_0(G_{K,h}(z))-V_0(z)
		\le-h\left(1-\frac{h\lambda}{2}\right)
		(\sin z)^T K\sin z.
	\end{equation}
	Therefore $V_0$ decreases strictly unless $\sin z=0$, and the fixed points
	are exactly $\{0,\pi\}^N$.
	At a fixed point $p$, the matrix $KD(p)$ is similar to
	$K^{1/2}D(p)K^{1/2}$.  By Sylvester's law of inertia, the number of
	positive eigenvalues equals the number of zero coordinates of $p$.
	All eigenvalues $\mu$ are nonzero and satisfy $|\mu|\le\lambda$, so the
	multipliers $1+h\mu$ lie off the unit circle: positive $\mu$ give
	unstable directions and negative $\mu$ stable ones.  The fixed points
	are therefore hyperbolic, with unstable dimension equal to their number
	of zero coordinates.
	
	For a non-fixed point, continuity gives a neighbourhood $U$ such that
	\[
	\sup_{y\in U}V_0(G_{K,h}(y))<\inf_{y\in U}V_0(y).
	\]
	Monotonicity then gives $G_{K,h}^k(U)\cap U=\varnothing$ for every $k\ge1$.
	Hence the nonwandering set is the stated finite fixed set.
	
	It remains to prove transversality.  The idea is to follow the largest
	coordinate of a decaying adjoint solution: diagonal dominance forces
	that coordinate into an invariant slab, where the adjoint equation
	cannot decrease it, which contradicts decay.
	
	Suppose that an intersection is not transverse.  Along its orbit
	$(z_n)_{n\in\mathbb Z}$, Lemma~\ref{lem:adjoint-transversality} supplies
	a nonzero solution tending to zero at both ends, with
	\[
		\eta_n=(I_N+hD_nK)\eta_{n+1},
		\qquad D_n=D(z_n).
	\]
	The maximum
	\[
		\rho=\max_{1\le i\le N,\ n\in\mathbb Z}|\eta_{i,n}|>0
	\]
	is attained.  Fix a time $j$ at which it is attained, and let
	$\mathcal I=\{i:|\eta_{i,j}|=\rho\}$.
	For $i\in\mathcal I$, put $s_i=\eta_{i,j}/\rho\in\{-1,1\}$.
	Diagonal dominance gives
	\[
	s_i(K\eta_j)_i
	\ge\rho\left(k_{ii}-\sum_{\ell\ne i}|k_{i\ell}|\right)\ge0.
	\]
	For at least one $i\in\mathcal I$ the left side is strictly positive.
	Otherwise, equality forces $k_{i\ell}=0$ for every
	$i\in\mathcal I$ and $\ell\notin\mathcal I$: a nonzero such entry
	would multiply a coordinate of absolute value strictly less than $\rho$.
	It would follow that
	$K_{\mathcal I\mathcal I}(\eta_j)_{\mathcal I}=0$,
	contradicting positive definiteness of this principal submatrix.
	
	Choose such an index $i$ and put $s=s_i$.
	Write $c_{i,n}=\cos z_{i,n}$.  The adjoint equation at time $j-1$ gives
	\[
	s\eta_{i,j-1}
	=\rho+hc_{i,j-1}s(K\eta_j)_i\le\rho,
	\]
	so $c_{i,j-1}\le0$.  Lemma~\ref{lem:path-slabs} implies
	$c_{i,n}\le0$ for every $n\ge j-1$.
	Set $c=c_{i,j}\le0$.  The diagonal coefficient $1+hk_{ii}c$ is strictly
	positive, and the adjoint equation at time $j$ gives
	\begin{align*}
		\rho
		&=(1+hk_{ii}c)s\eta_{i,j+1}
		+hc\sum_{\ell\ne i}k_{i\ell}s\eta_{\ell,j+1}\\
		&\le\rho\left[1+hc
		\left(k_{ii}-\sum_{\ell\ne i}|k_{i\ell}|\right)\right]
		\le\rho.
	\end{align*}
	Equality and the strictly positive diagonal coefficient force
	$\eta_{i,j+1}=s\rho$.  This step uses only $|\eta_{i,n}|=\rho$ and
	$c_{i,n}\le0$, so repeating it forces $\eta_{i,n}=s\rho$ for all
	$n\ge j$, contrary to decay.  This proves transversality.
	
	Finally, \eqref{eq:matrix-ms-descent}, compactness and finiteness of the
	critical set give convergence of every orbit to one fixed point by the
	argument of Proposition~\ref{thm:finite-convergence}.
	The unique minimum is $(\pi,\ldots,\pi)$.  All other fixed points have
	stable manifolds of dimension less than $N$, so their basins have measure
	zero, as in Theorem~\ref{thm:global-dynamics}.
\end{proof}

\begin{remark}
	\label{prop:sharpK}
	The optimality argument applies to every matrix $K$ similar to a
	symmetric one with $\lambda_{\max}(K)>0$: $G_{K,h}$ is not a local
	diffeomorphism when $h\ge1/\lambda_{\max}(K)$.
\end{remark}

\begin{remark}
	\label{rem:sink-threshold}
	Under the hypotheses of Theorem~\ref{thm:matrix-morse-smale}, for
	$h=1/\lambda_{\max}(K)$ the singular point found in the proof is the sink
	$(\pi,\ldots,\pi)$ itself, where $DG_{K,h}(\pi,  \ldots \pi)=I_N-hK$.  For
	$h>2/\lambda_{\max}(K)$ this point has the multiplier
	$1-h\lambda_{\max}(K)<-1$ and is no longer a sink.
\end{remark}

For a tree with positive weights $a_e$, the reduced map in edge
coordinates is $z\mapsto z+hB^TBA\sin z$
(Proposition~\ref{thm:general-tree-factorisation}), which has the form
$G_{K,h}$ only for equal weights.  We therefore extend
Theorem~\ref{thm:matrix-morse-smale} to maps
\[
G_{K,A,h}(z)=z+hKA\sin z\pmod{2\pi},\qquad
A=\operatorname{diag}(a_1,\ldots,a_N),\quad a_i>0.
\]

\begin{theorem}[Weighted matrix criterion]
	\label{thm:weighted}
	Let $K$ be symmetric, nonsingular and weakly diagonally dominant, and
	suppose that
	\begin{equation}
		\label{eq:weighted-dominance}
		k_{ii}a_i\ge\sum_{j\ne i}|k_{ij}|a_j,\qquad i=1,\ldots,N.
	\end{equation}
	Put $\tilde\lambda=\lambda_{\max}(A^{1/2}KA^{1/2})$.  For
	$0<h<1/\tilde\lambda$, the map $G_{K,A,h}$ is a Morse--Smale
	diffeomorphism with the nonwandering set, unstable dimensions and
	full-measure sink basin of Theorem~\ref{thm:matrix-morse-smale}.
\end{theorem}

For $A=I_N$, condition \eqref{eq:weighted-dominance} is
\eqref{eq:matrix-dominance} and the theorem is
Theorem~\ref{thm:matrix-morse-smale}.  The translation $\tau$ of
Remark~\ref{rem:sinusoidal-sign-conjugacy} below conjugates $G_{K,A,h}$ to
$z\mapsto z-hKA\sin z$, so the theorem also covers the opposite sign.

\begin{proof}
	Put $\tilde K=A^{1/2}KA^{1/2}$, which is positive definite since $K$ is.
	   As $A$ and $D(z)=\operatorname{diag}(\cos z_1,\ldots,\cos z_N)$ are diagonal, $A^{1/2}KAD(z)A^{-1/2}=\tilde KD(z)$.
	Replacing $K$ by $\tilde K$ in the proof of
	Theorem~\ref{thm:matrix-morse-smale} therefore shows that
	$DG_{K,A,h}=I_N+hKAD(z)$ has eigenvalues in
	$[1-h\tilde\lambda,1+h\tilde\lambda]\subset(0,2)$, so that $G_{K,A,h}$
	is a diffeomorphism, and that each point of $\{0,\pi\}^N$ is hyperbolic
	with unstable dimension equal to its number of zero coordinates.  With
	$V_A(z)=\sum_ia_i\cos z_i$ and $u=A\sin z$, the bound
	$u^TKAKu\le\tilde\lambda\,u^TKu$ and Taylor's inequality give
	\[
	V_A(G_{K,A,h}(z))-V_A(z)\le
	-h\Bigl(1-\frac{h\tilde\lambda}{2}\Bigr)u^TKu,
	\]
	which replaces \eqref{eq:matrix-ms-descent}; the fixed points, the
	nonwandering set and the basin follow as before.
	
	The $i$th component of the map is $z_i+hk_{ii}a_i\sin z_i+hb$ with
	$|b|\le k_{ii}a_i$ by \eqref{eq:weighted-dominance}, and
	$hk_{ii}a_i=h\tilde K_{ii}<1$.  Hence the proof of
	Lemma~\ref{lem:path-slabs} applies, and the slabs $\mathcal S_i$ are
	forward invariant.  The adjoint equation is now
	$\eta_n=(I_N+hD_nAK)\eta_{n+1}$: its $i$th row is that of the proof of
	Theorem~\ref{thm:matrix-morse-smale}, with $hc_{i,n}$ multiplied by
	$a_i>0$.  The maximum argument uses \eqref{eq:matrix-dominance} for $K$,
	positive definiteness of $K$, and positivity of the diagonal coefficient
	$1+ha_ik_{ii}c$, so it applies unchanged and proves transversality.
\end{proof}

\section{Network geometry and structural stability}
\label{sec:graph-topology}
\label{sec:sinusoidal-networks}

On trees, edge differences are independent coordinates and the critical
points can be classified directly. For unit-weight paths and the star with three leaves,
Theorem~\ref{thm:matrix-morse-smale} also gives transversality, and
Theorem~\ref{thm:weighted} extends this to paths with discretely concave
weights.  We first establish these network cases, then consider
their perturbations and the limits of the criterion.

\subsection{Trees and separated potentials}
\label{subsec:general-trees}

Let $T=(\mathsf V_T,\mathsf E_T)$ be a tree with $N+1$ vertices, positive
weights $a_e$ and smooth $P$-periodic odd edge laws $g_e$, with edge
potentials $\Phi_e$ as in \eqref{eq:Phi-def}.  A tree has exactly $N$
edges, so the edge differences $z=B^T\theta$ lie in $\TP^N$.  We write
\begin{equation}
	V_T(z)=\sum_{e\in\mathsf E_T}a_e\Phi_e(z_e),
	\qquad K_T=B^TB.
	\label{eq:general-tree-potential}
\end{equation}

\begin{proposition}[Critical points on trees]
	\label{thm:general-tree-factorisation}
	The edge differences $z=B^T\theta$ identify the common-phase quotient
	with $\TP^N$.  In these coordinates the reduced potential is $V_T$, and
	the reduced map is
	\begin{equation}
		z\longmapsto z-hK_T\nabla V_T(z)\pmod P.
		\label{eq:general-tree-map}
	\end{equation}
	Its critical set is
	\begin{equation}
		\Crit(V_T)=\prod_{e\in\mathsf E_T}\Crit(\Phi_e).
		\label{eq:general-tree-critical}
	\end{equation}
	If every $\Phi_e$ is Morse, then $V_T$ is Morse, with
	\begin{equation}
		\operatorname{ind}_{\mathrm M}(z)
		=\#\{e\in\mathsf E_T:\Phi_e''(z_e)<0\}.
		\label{eq:general-tree-index}
	\end{equation}
	If, moreover, each $\Phi_e$ has a unique local minimum $m_e$ modulo $P$,
	then $m=(m_e)_{e\in\mathsf E_T}$ is the unique local and absolute minimum
	of $V_T$ and, for $0<h<1/\Lambda$, the unique sink of
	\eqref{eq:general-tree-map}, with a basin of full Lebesgue measure.
\end{proposition}

\begin{proof}
	Fix the phase of vertex $N+1$.  Every other vertex is joined to its
	parent by a unique edge, and its phase is recovered from its parent's by
	adding or subtracting that edge difference.  Hence $\theta\mapsto
	B^T\theta$ induces a bijection from the quotient onto $\TP^N$.  Since the
	inverse uses only additions and subtractions, the matrix $C=B^TS$ of
	\eqref{eq:reduced-matrices} has an integer inverse, so $\det C=\pm1$ (that is, $C$ is unimodular) and
	$z=Cx$ is a linear automorphism of $\TP^N$ that preserves Lebesgue
	measure.  This proves the coordinate claim, and $V_T(B^T\theta)
	=\mathcal V(\theta)$ gives the potential.  Applying $B^T$ to
	\eqref{eq:network} gives the reduced map.  The Hessian of $V_T$ is
	diagonal, with entries $a_e\Phi_e''(z_e)$, so the critical set, indices
	and minima follow coordinate by coordinate.  The dynamical conclusion is
	Theorem~\ref{thm:global-dynamics}, transported to the $z$ coordinates by
	the automorphism $z=Cx$.
\end{proof}

Thus the potential separates, while the dynamics remains coupled
through $K_T$.

For the sinusoidal law $g(s)=\sin s$, take $\Phi(s)=\cos s$ and
$P=2\pi$.  On any positive-weight graph, the potential and spectral bound
become
\begin{equation}
	\mathcal V(\theta)=\sum_{e\in\mathcal E}a_e\cos(b_e^T\theta),
	\qquad \Lambda=\lambda_{\max}(L_a),\qquad L_a=BAB^T,
	\label{eq:sin-general-potential}
\end{equation}
by Remark~\ref{cor:common-law}, so Theorem~\ref{thm:finite-estimates}
gives strict descent for $0<h<2/\lambda_{\max}(L_a)$ and invertibility
for $0<h<1/\lambda_{\max}(L_a)$.

\begin{corollary}[Sinusoidal trees]
	\label{thm:tree}
	For a positive-weight tree with sinusoidal law, the reduced potential
	\begin{equation}
		V_T(z)=\sum_{e\in\mathsf E_T}a_e\cos z_e,
		\qquad \Crit(V_T)=\{0,\pi\}^N,
		\label{eq:tree-edge-coordinates}
	\end{equation}
	is Morse.  A critical point with $r$ zero coordinates has index $r$, so
	there are $\binom Nr$ points of each index.  The unique minimum is the
	phase-opposition state
	\begin{equation}
		z_e=\pi\quad\text{for every edge }e,
		\label{eq:tree-opposition}
	\end{equation}
	which is the unique sink, with a basin of full measure, whenever
	\begin{equation}
		0<h<\frac1{\lambda_{\max}(L_a)}.
		\label{eq:tree-diffeo}
	\end{equation}
\end{corollary}

\begin{proof}
	Apply Proposition~\ref{thm:general-tree-factorisation} to
	$\Phi_e=\cos$, whose critical points $0$ and $\pi$ have second
	derivatives $-1$ and $1$, respectively; here
	$\Lambda=\lambda_{\max}(L_a)$ by \eqref{eq:sin-general-potential}.
\end{proof}

\begin{corollary}[Optimal step sizes on trees]
	\label{cor:tree-sharp}
	For a sinusoidal tree with positive weights, both bounds of
	Theorem~\ref{thm:finite-estimates} are optimal.  For $h\ge1/\Lambda$ the
	reduced map is not a local diffeomorphism.  For $h>2/\Lambda$ phase
	opposition is not a sink, and strict descent fails near it.
\end{corollary}

\begin{proof}
	By Proposition~\ref{thm:general-tree-factorisation}, the reduced map is
	$z\mapsto z+hB^TBA\sin z$.  Since $B$ has full column rank, $B^TBA$ is
	similar to the positive definite matrix $A^{1/2}B^TBA^{1/2}$, so
	$\lambda_{\max}(B^TBA)=\lambda_{\max}(L_a)=\Lambda>0$, and
	Remark~\ref{prop:sharpK} gives the first claim.  At $(\pi,\ldots,\pi)$
	the derivative is $I_N-hB^TBA$, with multiplier $1-h\Lambda<-1$ for
	$h>2/\Lambda$; a second-order expansion of $V_T$, whose Hessian there
	is $A$, shows that $V_T$ increases along the corresponding eigendirection.
\end{proof}

Identifying phases $0,\pi$ with spins $\sigma_i=+1,-1$ gives
$\cos(\theta_i-\theta_j)=\sigma_i\sigma_j$.  The restricted potential is
therefore the positive-weight antiferromagnetic Ising energy
\cite{Barahona1982}, whose minimisation is the weighted Max-Cut problem:
find the bipartition of the vertices that maximises the total weight of
edges between the two classes \cite{GoemansWilliamson1995}.  A tree is
bipartite \cite{Diestel2017}, so assigning the two phases to its two
vertex classes places every edge between the classes and attains the
minimum of every cosine term.  On trees this optimisation problem is
therefore trivial; what Corollary~\ref{thm:tree} adds is that the
oscillator dynamics finds its solution: for $h$ in \eqref{eq:tree-diffeo},
phase opposition attracts almost every reduced initial condition, and
every other fixed point has an unstable direction.  Trees thus provide
an exactly solvable test case for oscillator-based energy minimisation
\cite{WangWuNobelRoychowdhury2021,
	ChengShuklaLin2025GraphOIM,Steinerberger2023MaxCut}.

%
%
%

\begin{figure}[htbp]
\centering
\includegraphics[width=0.82\textwidth]{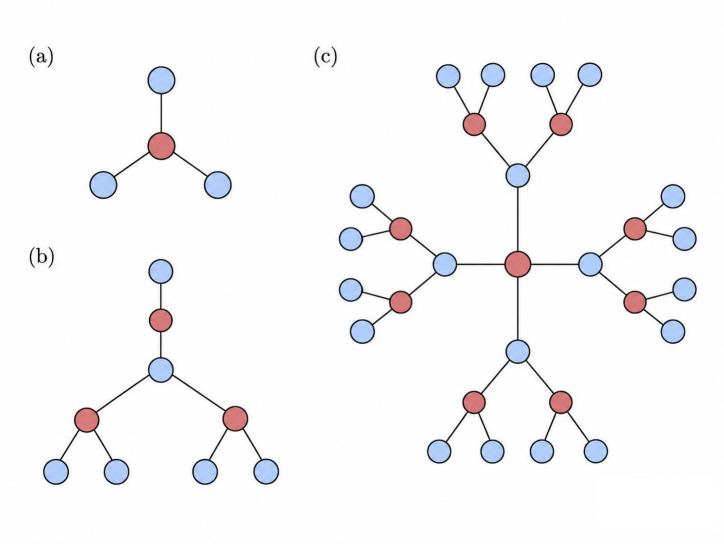}
\caption{Phase opposition on three trees.  Assigning phases $0$ and
$\pi$ to the two vertex classes places every edge at the minimum of its
sinusoidal potential.  The bipartition is unique up to interchange of
the classes, and the phase pattern is unique modulo common rotation.}
\label{fig:tree-bipartitions}
\end{figure}

\subsection{Paths and the star with three leaves}

For the unit-weight path on $N+1$ vertices, orient edge $i$ from $i$ to
$i+1$.  Its edge coordinates and potential are
\[
	z_i=\theta_{i+1}-\theta_i\pmod{2\pi},\quad i=1,\ldots,N,
	\qquad V_N(z)=\sum_{i=1}^N\cos z_i.
\]
The reduced map is $G_{K_N,h}$ from \eqref{eq:matrix-ms-map}, with
\begin{equation}
	\begin{aligned}
		G_{K_N,h}(z)&=z-hK_N\nabla V_N(z)=z+hK_N\sin z\pmod{2\pi},\\
		K_N&=B^TB=
		\begin{pmatrix}
			2&-1&0&\cdots&0\\
			-1&2&-1&\ddots&\vdots\\
			0&\ddots&\ddots&\ddots&0\\
			\vdots&\ddots&-1&2&-1\\
			0&\cdots&0&-1&2
		\end{pmatrix}.
	\end{aligned}
	\label{eq:path-gradient}
\end{equation}
Its eigenvalues are the nonzero path-Laplacian eigenvalues
\cite{Chung1997}:
\[
	\lambda_k=2-2\cos\frac{k\pi}{N+1},\quad k=1,\ldots,N,
	\qquad\lambda_{\max}(K_N)=2+2\cos\frac{\pi}{N+1}.
\]
Since $B^TB$ and $L_a=BB^T$ have the same nonzero spectrum,
$\Lambda=\lambda_{\max}(K_N)$, and Theorem~\ref{thm:finite-estimates}
shows that $V_N$ decreases strictly outside its critical set for
\[
	0<h<\frac{2}{2+2\cos(\pi/(N+1))}.
\]

\begin{corollary}[Sinusoidal paths]
	\label{thm:path}\label{thm:path-morse-smale}
	For every $N\ge1$, the reduced path map $G_{K_N,h}$ is a Morse--Smale
	diffeomorphism for
	\begin{equation}
		0<h<\frac{1}{2+2\cos(\pi/(N+1))};
		\label{eq:path-diffeo-range}
	\end{equation}
	in particular, $0<h\le1/4$ suffices for every path length.  Its fixed
	points are the $2^N$ critical points of Corollary~\ref{thm:tree};
	$(\pi,\ldots,\pi)$ is the unique sink, with a full-measure basin, and
	$(0,\ldots,0)$ is the unique source.
\end{corollary}

\begin{proof}
	The matrix $K_N$ is nonsingular, since $B$ has full column rank, and
	weakly diagonally dominant, so Theorem~\ref{thm:matrix-morse-smale}
	applies.  Since $\lambda_{\max}(K_N)<4$ for finite $N$, $h=1/4$ is
	included.
\end{proof}

For $N=1$ this recovers the two-oscillator model
\cite{OliveiraMelo2015}; for $N=2$ it gives the aligned
three-oscillator model \cite{BuescuDAnielloOliveira2025Aligned}.
In the latter case, reversing one edge coordinate conjugates $K_2$ by
$\operatorname{diag}(1,-1)$, changing its off-diagonal entries from
$-1$ to $1$ and recovering the coordinates used in
\cite{BuescuDAnielloOliveira2025Aligned,
	BuescuDAnielloOliveira2025Lyapunov}.

\begin{figure}[htbp]
	\centering
	\includegraphics[width=0.61\textwidth]{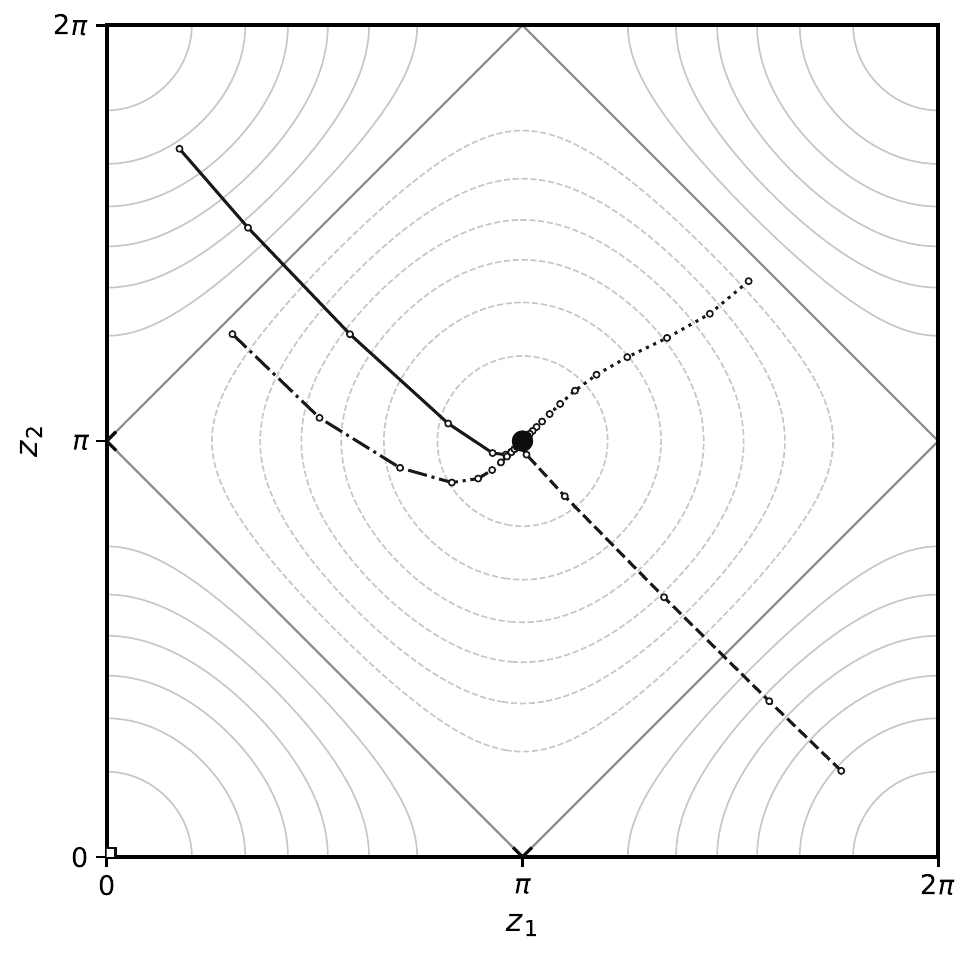}
	\par\medskip
	\includegraphics[width=0.90\textwidth]{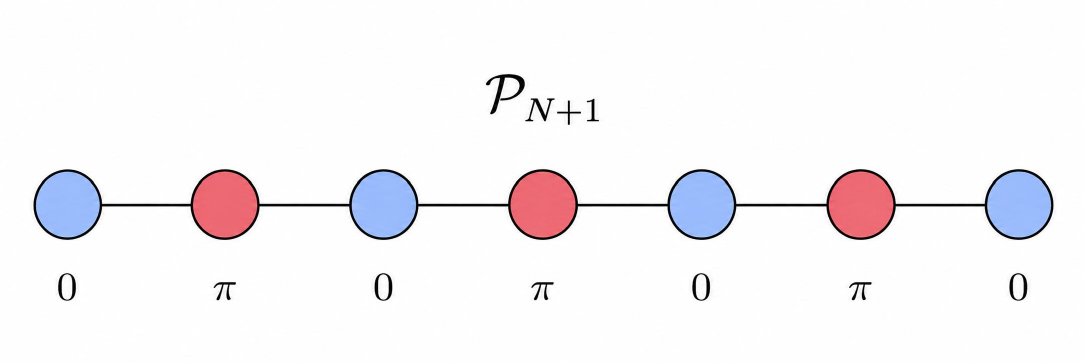}
	\caption{Top: representative discrete orbits on the three-oscillator
		path, with $h=0.28<1/3$, over level sets of
		$V_2=\cos z_1+\cos z_2$.  The open square is the source $(0,0)$, the
		crosses are the saddles $(0,\pi),(\pi,0)$, and the filled circle is the
		sink $(\pi,\pi)$.  Bottom: phase opposition on a path, with alternating
		vertex phases $0,\pi$ and every edge difference equal to $\pi$.}
	\label{fig:path-discrete-dynamics}\label{fig:path-opposition}
\end{figure}

\begin{corollary}[The star with three leaves]
	\label{cor:star-morse-smale}
	For four oscillators on the unit-weight star $\mathsf K_{1,3}$, with one centre
	joined to three leaves, the reduced sinusoidal map is Morse--Smale for
	$0<h<1/4$.  It has eight fixed points: one sink, three of unstable
	dimension one, three of unstable dimension two, and one source.
\end{corollary}

\begin{proof}
	Orienting all edges from the centre gives the weakly diagonally dominant
	matrix
	\[
		K=B^TB=\begin{pmatrix}2&1&1\\1&2&1\\1&1&2\end{pmatrix},
		\qquad \operatorname{spec}(K)=\{1,1,4\},
	\]
	which is nonsingular. The result follows from direct application of  Theorem~\ref{thm:matrix-morse-smale}.
\end{proof}

For a unit-weight tree, an edge $\{u,v\}$ meets
$\deg(u)+\deg(v)-2$ other edges.  Hence $B^TB$ is weakly diagonally
dominant exactly when $\deg(u)+\deg(v)\le4$ on every edge.  Such a
connected tree is a path or $\mathsf K_{1,3}$: a vertex of degree three can have
only leaves as neighbours, and larger degrees are excluded.  This
classification does not depend on the choice of coordinates: the only
integer changes of coordinates that preserve the form $z+hK\sin z$ are signed permutations, since
$\sin(\sum_j c_jw_j)$ is a linear combination of the $\sin w_j$ only when
exactly one $c_j$ is nonzero and equal to $\pm1$; signed permutations
preserve diagonal dominance.  It classifies
the unit-weight trees covered by the matrix criterion; it does not
exclude Morse--Smale dynamics on other trees.

\begin{corollary}[Weighted paths]
	\label{cor:wpath}
	Consider the path on $N+1$ vertices with weights $a_1,\ldots,a_N>0$ and
	sinusoidal interaction, and set $a_0=a_{N+1}=0$.  If
	\begin{equation}
		\label{eq:concave}
		2a_i\ge a_{i-1}+a_{i+1},\qquad i=1,\ldots,N,
	\end{equation}
	then the reduced map is Morse--Smale for $0<h<1/\lambda_{\max}(L_a)$,
	and this interval is optimal.
\end{corollary}

\begin{proof}
	By Proposition~\ref{thm:general-tree-factorisation}, the reduced map is
	$G_{K_N,A,h}$, with $K_N$ as in \eqref{eq:path-gradient}, and
	\eqref{eq:concave} is \eqref{eq:weighted-dominance}.  The matrix
	$A^{1/2}K_NA^{1/2}=(BA^{1/2})^T(BA^{1/2})$ has the nonzero spectrum of
	$BAB^T=L_a$, so $\tilde\lambda=\lambda_{\max}(L_a)$.
	Theorem~\ref{thm:weighted} applies, and optimality is
	Corollary~\ref{cor:tree-sharp}.
\end{proof}

\begin{remark}
	The weights $a_i=\sin\bigl(i\pi/(N+1)\bigr)$ satisfy \eqref{eq:concave}
	strictly, the weights $a_i=i$ satisfy it with equality for $i<N$, and
	geometric weights $a_i=r^i$, $r\ne1$, violate it.  For paths the
	admissible weights thus form a closed convex cone with nonempty
	interior, an explicit open set of Morse--Smale weights.  Unit weights lie
	on its boundary, so this complements the non-explicit neighbourhood of
	Corollary~\ref{cor:hamiltonian-morse-smale}.  Strict dominance of $K$ in
	some row is essential: by symmetry of $K$,
	\[
	\sum_{i=1}^N\Bigl(k_{ii}a_i-\sum_{j\ne i}|k_{ij}|a_j\Bigr)
	=\sum_{j=1}^Na_j\Bigl(k_{jj}-\sum_{i\ne j}|k_{ij}|\Bigr),
	\]
		so if \eqref{eq:matrix-dominance} is an equality in every row, so is
	\eqref{eq:weighted-dominance}.  For $\mathsf K_{1,3}$ this forces
	$a_1=a_2=a_3$, which only rescales $h$ in
	Corollary~\ref{cor:star-morse-smale}.
	\end{remark}

\subsection{Structural stability and additional links}

\begin{corollary}[Structural stability]
	\label{cor:matrix-structural-stability}
	\label{cor:path-structural-stability}
	Fix $K$ and $h$ satisfying the hypotheses of
	Theorem~\ref{thm:matrix-morse-smale}.  There is a $C^1$ neighbourhood
	$\mathcal U$ of $G_{K,h}$ in the space of $C^1$ diffeomorphisms of
	$\mathbb T^N$ such that every $\widehat G\in\mathcal U$ is Morse--Smale
	and topologically conjugate to $G_{K,h}$, that is,
	$H\circ G_{K,h}=\widehat G\circ H$ for some homeomorphism
	$H:\mathbb T^N\to\mathbb T^N$.  Each such map has exactly $2^N$ fixed
	points, $\binom Nr$ of unstable dimension $r$, and one sink.
\end{corollary}

\begin{proof}
	Apply the structural stability theorem for Morse--Smale diffeomorphisms
	\cite{PalisSmale1970} to Theorem~\ref{thm:matrix-morse-smale}.
	Hyperbolic continuation preserves the unstable dimensions of the fixed
	points.
\end{proof}

In the network cases, this includes small changes in edge weights,
interaction laws and oscillator frequencies whenever the induced reduced
map is $C^1$-close to the original map.  The perturbation need not preserve
the potential or reciprocal coupling.  The full model must still commute
with common phase shifts to define a map on the same quotient.
The admissible perturbation size may depend on $K$, $N$ and $h$.

\begin{corollary}[Graphs containing a spanning path]
	\label{cor:hamiltonian-morse-smale}
	Let $\mathcal G$ be an undirected graph on $N+1$ vertices containing a
	spanning path $\Pi$. Fix $h$ satisfying \eqref{eq:path-diffeo-range}.  There is
	$\varepsilon\in(0,1)$ such that the reduced sinusoidal map on
	$\mathcal G$ is Morse--Smale and topologically conjugate to the
	unit-weight path map whenever
	\begin{equation}
		\label{eq:hamiltonian-weight-range}
		|a_e-1|<\varepsilon\quad(e\in\Pi),
		\qquad
		0<a_e<\varepsilon\quad(e\notin\Pi).
	\end{equation}
	In particular, these positive weights form an open set of parameters with
	exactly $2^N$ fixed points and one sink.
\end{corollary}

\begin{proof}
	Set $a_e=1$ on $\Pi$ and $a_e=0$ on the remaining edges.  This is the
	path map of Corollary~\ref{thm:path-morse-smale}, expressed in the
	reference-phase coordinates of Section~\ref{sec:reciprocal-gradient};
	by Proposition~\ref{thm:general-tree-factorisation} the two coordinate
	systems differ by a linear automorphism of the torus, so the map is
	Morse--Smale in either.  The reduced map depends continuously in $C^1$ on
	all edge weights, and
	Corollary~\ref{cor:matrix-structural-stability} therefore supplies
	$\varepsilon$ with the stated property.
\end{proof}

A spanning path is also called a Hamiltonian path. 

This corollary includes cycles closed by a weak link and complete graphs
whose additional links are sufficiently weak.  The value of
$\varepsilon$ may depend on the graph, the chosen path and $h$.

\begin{remark}[The two signs of sinusoidal interaction]
	\label{rem:sinusoidal-sign-conjugacy}
	Let $\tau(z)=z+\pi\one_N\pmod{2\pi}$.  Then
	\[
	\tau^{-1}\circ G_{K,h}\circ\tau(z)=z-hK\sin z.
	\]
	Consequently, the matrix criterion also applies with the opposite sign
	of the sine interaction.  On a tree, $\tau$ is induced by adding $\pi$
	to the vertex phases in one bipartition class.
	The path and three-leaf-star results therefore also give structural
	stability for $g(s)=-\sin s$, whose unique sink is the equal-phase state.
	The weak-link corollary holds for that sign too, by perturbing the
	corresponding path map.
\end{remark}

All structural stability statements concern reduced dynamics.  The full
phase system retains the neutral common-phase direction and is not
Morse--Smale.  For continuous all-to-all Kuramoto dynamics, Dai, Fiedler
and L\'opez-Nieto \cite{DaiFiedlerLopezNieto2025} prove transversality for
heteroclinic connections between equilibria with positive order parameter.
Theorem~\ref{thm:matrix-morse-smale} gives a direct finite-step criterion
on the entire reduced torus.

\subsection{A family of trees with non-transverse intersections}

Even for unit weights, a Morse potential and strict Lyapunov decrease do
not guarantee transversality on every tree.

\begin{proposition}[A tree obstruction to transversality]
	\label{prop:tree-sign-obstruction}
	Let $T$ be a tree with unit weights.  Suppose that its edges admit signs
	$d_e\in\{-1,1\}$, including a negative sign, such that
	\begin{equation}
		\label{eq:tree-sign-condition}
		\sum_{e\ni v}d_e=1\qquad\text{at every vertex }v.
	\end{equation}
	Then, for every $0<h<1/\lambda_{\max}(B^TB)$, its reduced sinusoidal map
	has a non-transverse heteroclinic intersection, and so is not
	Morse--Smale.  If $r$ edges have negative signs, the tree has $4r+2$
	vertices, and the stable and unstable tangent spaces along this
	intersection coincide, with dimension $3r+1$ in the reduced space of
	dimension $4r+1$.
\end{proposition}

\begin{proof}
	Orient every edge from one bipartition class to the other and set
	$K=B^TB$ and $d=(d_e)_{e\in\mathcal E}$.  The diagonal entries of $K$ are
	two, and its off-diagonal entries are one when the edges share a vertex
	and zero otherwise.  For an edge $uv$, \eqref{eq:tree-sign-condition}
	gives
	\[
	(Kd)_{uv}=\sum_{e\ni u}d_e+\sum_{e\ni v}d_e=2.
	\]
	Put $\Delta=\operatorname{diag}(d_e)$ and let $p_e=0$ if $d_e=1$, and
	$p_e=\pi$ if $d_e=-1$.  Writing $m$ for the number of edges and
	$\one_m=(1,\ldots,1)^T$, the circle $z(s)=p+s\one_m$ is invariant,
	with induced map $s\mapsto s+2h\sin s$.
	The bound on $h$ implies $h<1/2$, since $\lambda_{\max}(K)\ge2$.
	For $0<s<\pi$, the scalar map is strictly increasing and maps this
	interval into itself.  Since $s+2h\sin s>s$ there, every orbit has
	forward limit $\pi$ and backward limit zero.
	Thus the arc joins $p$ to its binary complement
	$q=p+\pi\one_m\pmod{2\pi}$.
	
	Along the arc, the derivative is $I_m+h\cos s\,K\Delta$.
	The matrix $K\Delta$ is similar to the symmetric matrix
	$K^{1/2}\Delta K^{1/2}$.  Sylvester's law of inertia shows that
	$K\Delta$ has $m-r$ positive and $r$ negative eigenvalues.  Every
	eigenvalue $\mu$ satisfies $|\mu|\le\lambda_{\max}(K)$.  All variational
	factors $1+h\mu\cos s$ are therefore positive, and the same real
	eigenspaces are preserved along the orbit.  Let $E_+$ be the sum of the
	positive eigenspaces of $K\Delta$.
	
	Tangent vectors to $W^u(p)$ decay under backward iteration of the
	derivative.  A nonzero component in a negative eigenspace instead grows
	backwards, because its forward multiplier tends to $1+h\mu\in(0,1)$
	near $p$.  Hence $T_zW^u(p)\subseteq E_+$, and equality follows from the
	dimensions.  Similarly, negative eigendirections grow forwards near $q$,
	where their multipliers tend to $1-h\mu>1$.  Thus
	\begin{equation}
		\label{eq:tree-obstruction-tangents}
		T_zW^u(p)=T_zW^s(q)=E_+.
	\end{equation}
	A negative sign makes $E_+$ a proper subspace, so the intersection is
	not transverse.
	
	Finally, summing \eqref{eq:tree-sign-condition} over vertices gives
	$2(m-2r)=m+1$, since a tree has $m+1$ vertices.  Therefore $m=4r+1$,
	and Sylvester's law gives $\dim E_+=m-r=3r+1$.
\end{proof}

The failure of transversality is not forced by dimension.  Since
$\dim W^u(p)=\dim W^s(q)=3r+1$ in a space of dimension $4r+1$, a
transverse intersection would have dimension $2r+1\ge1$ and is therefore
not excluded a priori; instead the two tangent spaces coincide.
For $h\ge1/\lambda_{\max}(B^TB)$ the map is not even a local
diffeomorphism, by Remark~\ref{prop:sharpK}.  These trees are therefore never
Morse--Smale.

\begin{proposition}[Six-vertex tree counterexample]
\label{prop:tree-nontransverse}
Consider the tree with two adjacent vertices of degree three, each joined
to two leaves, with unit weights and sinusoidal interaction.
For
\begin{equation}
\label{eq:tree-counterexample-range}
    0<h<\frac{2}{5+\sqrt{17}},
\end{equation}
its reduced map on $\mathbb T^5$ is a diffeomorphism with a non-transverse
heteroclinic intersection.
\end{proposition}

\begin{proof}
Label the vertices $1,\ldots,6$ and orient and order the edges as
\[
    (1,2),\quad(1,3),\quad(1,4),\quad(5,2),\quad(6,2),
\]
where the first vertex is the tail and the second is the head.
In edge coordinates,
\begin{equation}
\label{eq:tree-counterexample-map}
    G_{K_*,h}(z)=z+hK_*\sin z\pmod{2\pi},\qquad
    K_*=B^TB=
    \begin{pmatrix}
      2&1&1&1&1\\
      1&2&1&0&0\\
      1&1&2&0&0\\
      1&0&0&2&1\\
      1&0&0&1&2
    \end{pmatrix}.
\end{equation}
The signs $d=(-1,1,1,1,1)^T$ satisfy
\eqref{eq:tree-sign-condition}.  The eigenvalues of $K_*$ are
$1,1,3,(5-\sqrt{17})/2$ and $(5+\sqrt{17})/2$, so
Proposition~\ref{prop:tree-sign-obstruction} applies in the stated range.
Explicitly, with
\[
    p=(\pi,0,0,0,0)^T,\qquad q=(0,\pi,\pi,\pi,\pi)^T,
\]
the arc $z(s)=p+s\one_5$, $0<s<\pi$, has scalar dynamics
\begin{equation}
\label{eq:tree-counterexample-circle}
    s(n+1)=s(n)+2h\sin s(n).
\end{equation}
For $\Delta_*=\operatorname{diag}(-1,1,1,1,1)$, the sum of the eigenspaces
of $K_*\Delta_*$ corresponding to positive eigenvalues is
\[
    E_+=\{\xi\in\R^5:
               4\xi_1=\xi_2+\xi_3+\xi_4+\xi_5\}.
\]
It equals both tangent spaces along this heteroclinic arc by
\eqref{eq:tree-obstruction-tangents}.
\end{proof}

These connections can be displayed on the invariant two-dimensional torus
\[
    \mathbb T^2_*=\{(u,v,v,v,v):(u,v)\in\mathbb T^2\}.
\]
The restriction of \eqref{eq:tree-counterexample-map} to $\mathbb T^2_*$ is
\[
    \begin{aligned}
        u^+&=u+h(2\sin u+4\sin v),\\
        v^+&=v+h(\sin u+3\sin v),
    \end{aligned}
    \qquad\pmod{2\pi},
\]
where the superscript $+$ denotes one iterate.
Write $W^u_*(p)$ and $W^s_*(q)$ for the unstable
and stable manifolds of this restricted map.  On $u=v+\pi$, put $s=v$.
Then \eqref{eq:tree-counterexample-circle} holds, and oddness gives the
reflected branch $-\pi<s<0$.  Both branches connect $p$ to $q$, since
$s=-\pi$ and $s=\pi$ represent the same point on the torus.
\begin{figure}[tbp]
\centering
\includegraphics[width=0.48\textwidth,trim=8bp 15bp 274bp 24bp,clip]{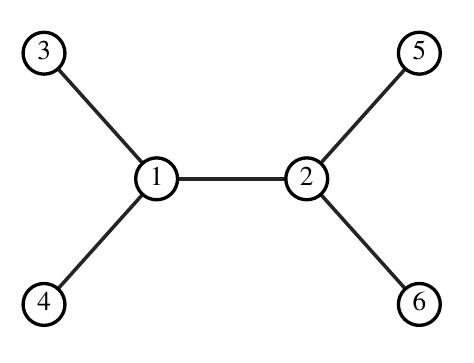}
\caption{Six-vertex tree counterexample. Undirected graph with unit weights.}
\label{fig:six-vertex-tree}
\end{figure}

Condition \eqref{eq:tree-sign-condition} says that exactly
$(\deg v-1)/2$ of the edges at each vertex $v$ are negative.  All
degrees are therefore odd, and when they are at most three the negative
edges form a perfect matching of the internal vertices.  In particular,
$\mathsf K_{1,3}$ admits no such signs, in agreement with
Corollary~\ref{cor:star-morse-smale}.

There are arbitrarily large trees satisfying
\eqref{eq:tree-sign-condition}.  Take a path of $2r$ internal vertices,
$r\ge1$, and attach leaves until each internal vertex has degree three.
Assign negative signs to alternate internal edges, pairing the internal
vertices consecutively, and positive signs to all other edges.
Each internal vertex then has one negative and two positive incident
edges, while every leaf has one positive edge.  This gives counterexamples
with $4r+2$ vertices; $r=1$ is the preceding six-vertex tree.
All retain the convergence and unique phase-opposition sink of
Corollary~\ref{thm:tree}.  Their obstruction is the non-transverse geometry
of invariant manifolds.


\FloatBarrier

\subsection{Complete graphs and cycle constraints}

Complete graphs illustrate how cycles can change the critical-set
geometry.  For the unit-weight complete graph $\mathsf K_{N+1}$ with sinusoidal
interaction, define
\begin{equation}
	Z(\theta)=\sum_{j=1}^{N+1}e^{\mathrm i\theta_j},
	\qquad \mathrm i^2=-1.
	\label{eq:Z-def}
\end{equation}
The normalised quantity $|Z|/(N+1)$ is the classical Kuramoto order
parameter \cite{Kuramoto1984,PikovskyRosenblumKurths2001,
	ScardoviSarletteSepulchre2007}.  The Laplacian of $\mathsf K_{N+1}$ is
$(N+1)I_{N+1}-\one\one^T$, whose nonzero eigenvalue $N+1$ has
multiplicity $N$.  Hence $\Lambda=N+1$, and
Theorem~\ref{thm:finite-estimates} gives strict descent for
$0<h<2/(N+1)$ and invertibility of the full and reduced maps for
$0<h<1/(N+1)$.

\begin{proposition}[All-to-all interaction]
	\label{thm:all-to-all}
	The intrinsic potential is
	\begin{equation}
		\mathcal V_{\mathrm{all}}(\theta)
		=\sum_{1\le i<j\le N+1}\cos(\theta_i-\theta_j)
		=\frac12|Z(\theta)|^2-\frac{N+1}{2}.
		\label{eq:all-potential-order}
	\end{equation}
	Its absolute minima are precisely the \emph{balanced} configurations,
	those with $Z(\theta)=0$, and their potential value is $-(N+1)/2$.
\end{proposition}

\begin{proof}
	Expand $|Z|^2$ to obtain \eqref{eq:all-potential-order}.  Equally spaced
	phases have $Z=0$, so the squared-modulus formula identifies the minimum
	and all its minimisers.
\end{proof}

For three labelled oscillators, fixing $\theta_3=0$ makes the remaining
two unit vectors complete an equilateral triangle.  The reduced minima
are therefore exactly
\[
	x=\left(\frac{2\pi}{3},\frac{4\pi}{3}\right),
	\qquad x=\left(\frac{4\pi}{3},\frac{2\pi}{3}\right).
\]
These are the attracting states of the three-oscillator all-to-all model
\cite{DAnielloOliveira2024,BuescuOliveira2026MorseSmale}; the bound
$h<1/3$ agrees with the parameter range of its direct analysis.
For $N\ge3$, the balanced set contains positive-dimensional families
even after removal of the common phase.  The reduced potential is
degenerate there, so the Morse and finite-critical-set conclusions do
not apply.

For general $P$-periodic edge laws, an assignment $z=(z_e)$ comes from
vertex phases exactly when it closes around every cycle.  Orient each
cycle $\mathcal C$ and let $\sigma_{\mathcal C e}=1$ if edge $e$ agrees
with that orientation, and $-1$ otherwise.  The compatibility condition is
\begin{equation}
	\sum_{e\in\mathcal C}\sigma_{\mathcal C e}z_e\equiv0\pmod P
	\qquad\text{for every cycle }\mathcal C.
	\label{eq:cycle-compatibility}
\end{equation}
Necessity follows by summing phase differences.  For sufficiency, assign
phases along a spanning tree; the cycle conditions give the prescribed
difference on every remaining edge.

Choose a global minimum $m_e$ of each $\Phi_e$.  All edge minima can
be attained simultaneously exactly when some such choice satisfies
\eqref{eq:cycle-compatibility}.  The network is then called
\emph{unfrustrated}, and otherwise \emph{frustrated}
\cite{Toulouse1977}.  Trees have no cycle constraints.  An odd cycle
with sinusoidal interaction is the simplest frustrated graph, since
assigning $\pi$ to every edge violates closure.  The potential and
finite-step bounds remain valid on graphs with cycles, but their
critical points must satisfy \eqref{eq:cycle-compatibility} and need
not minimise every edge term.

Absence of frustration does not by itself reproduce the tree picture.
With sinusoidal interaction every bipartite graph is unfrustrated, and
adding $\pi$ to the phases of one bipartition class conjugates the
dynamics to that of attractive coupling, as in
Remark~\ref{rem:sinusoidal-sign-conjugacy}.  For attractive coupling on
the cycle $C_n$, the twisted states $\theta_j=2\pi qj/n$ are
stable local minima of the attractive potential $-\sum\cos(\theta_i-\theta_j)$ whenever
 $2\pi|q|/n<\pi/2$ \cite{WileyStrogatzGirvan2006}.
Hence every even cycle with $n\ge6$ is unfrustrated, yet has local minima
other than phase opposition; see also \cite{LingXuBandeira2019} for the
landscape of such networks.  On trees, by contrast,
Corollary~\ref{thm:tree} shows that phase opposition is the only local
minimum.

\section{Scope of applicability and physical connections}
\label{sec:applications}

The results apply to any discrete phase update of the form
\eqref{eq:network}: the Huygens-type models that motivate this work
\cite{OliveiraMelo2015,DAnielloOliveira2024,BuescuDAnielloOliveira2025Aligned},
and sampled or impulsive phase controllers with positive symmetric weights
and smooth periodic odd interactions.  For higher-dimensional oscillators,
a phase reduction must first produce such a map
\cite{Nakao2016PhaseReduction}.
Corollary~\ref{cor:matrix-structural-stability} then extends the path and
three-leaf-star conclusions to sufficiently small perturbations, provided
the reduced phase space is unchanged.

Several physical systems share this pairwise structure but are usually
modelled in continuous time; our results reach them only through a
discrete update of the form \eqref{eq:network}.  Resistively coupled
Wien-bridge oscillators admit an effective Sakaguchi--Kuramoto
description whose symmetric zero-phase-lag limit has odd sinusoidal
interaction \cite{EnglishZengMertens2015}, with a coupling sign that must
be matched to ours.  Oscillator-based Ising machines encode binary states
in relative phases and use symmetric couplings to realise an Ising-type
energy \cite{WangWuNobelRoychowdhury2021}; the trees of
Section~\ref{subsec:general-trees} provide an exactly solvable test case,
although coupling architecture, injection locking, bias terms, circuit
non-idealities and even components of the realised interaction can alter
the intended energy dynamics
\cite{MoyEtAl2022,LoEtAl2023,GraberHofmann2024,ChengDaiStanLin2026Delayed}.
Nearest-neighbour arrays of MEMS oscillators
\cite{BhaskarShayakRandZehnder2021} are natural candidates for the path
results, once amplitude dynamics, detuning and forcing are accounted for.
Finally, radial electrical networks are trees, and network-reduced
power-system models can be related to nonuniform Kuramoto oscillators
\cite{DorflerBullo2012Power}.  In our convention, lossless coupling
corresponds to $g(s)=-\sin s$, whose edge potential $-\cos s$ is
minimised at equal phase, and
Remark~\ref{rem:sinusoidal-sign-conjugacy} relates it to the sign treated
here.  Heterogeneous injections, losses, different time constants and
control loops nevertheless prevent a direct identification of complete
power-system models with \eqref{eq:network}.

\section{Conclusions}
\label{sec:conclusion}

The matrix criterion establishes Morse--Smale dynamics for
$z\mapsto z+hK\sin z$ when $K$ is symmetric, nonsingular and weakly
diagonally dominant (and thus positive definite) throughout $0<h<1/\lambda_{\max}(K)$.
 This interval is optimal.
It covers unit-weight paths of every finite length and the three-leaf star. 
A weighted version extends it to paths with discretely concave weights.
Structural stability then supplies open sets of positive weights on every
graph containing a spanning path, with sufficiently weak additional links.
These conclusions also hold for the opposite sinusoidal sign.
The permitted perturbation size depends on the fixed map, including $h$.

The tree counterexamples show why transversality requires a separate
condition. For every $r\ge1$, a unit-weight tree with $4r+2$ vertices has
a non-transverse heteroclinic intersection throughout its stated
diffeomorphism interval. Its potential is Morse, and its unique sink
still attracts almost every orbit. Convergence and hyperbolicity therefore
do not imply the Morse--Smale property.

For general reciprocal networks, the intrinsic potential and the weighted
Laplacian bound $\Lambda$ in~\eqref{eq:Lambda} give two sufficient intervals.
For $0<h<2/\Lambda$, the potential decreases strictly away from its
critical set, and every reduced orbit approaches that set; finiteness
ensures convergence to one fixed point. For $0<h<1/\Lambda$, the maps
are global diffeomorphisms and, if the reduced potential is Morse, almost
every orbit converges to a minimum.  On sinusoidal trees this minimum is
phase opposition, the ground state of the associated antiferromagnetic
Ising energy.

The Morse--Smale and structural stability conclusions concern the dynamics
after removal of the common phase. The full map retains a neutral
direction. Small perturbations may break reciprocity or the potential
structure, but must preserve common-phase symmetry to act on the same
quotient. Delays and amplitude dynamics require separate analysis.

Several questions remain open.  Neither the criterion nor the obstruction
decides the unit-weight stars $\mathsf K_{1,n}$ with $n\ge4$, which are not
diagonally dominant and admit no signs satisfying
\eqref{eq:tree-sign-condition}, nor trees with generic weights beyond the concave path weights of Corollary~\ref{cor:wpath}.  The
counterexamples rely on an invariant circle that exists for unit weights;
whether their non-transversality persists under small changes of the
weights is not known.  Finally, we expect the maximum argument of
Section~\ref{sec:morse-smale} to apply also to the adjoint of the gradient
flow $\dot z=K\sin z$, giving a continuous-time counterpart of the
criterion.

\section*{Declarations}

\paragraph{Data availability.}
No datasets were generated or analysed during the preparation of this article.
All proofs and calculations are presented in the paper.

\paragraph{Author contributions.}
All authors contributed equally to all stages of the work. All authors reviewed
and approved the final manuscript.

\paragraph{Funding.}
Jorge Buescu acknowledges partial support from Funda\c{c}\~ao para a Ci\^encia e a
Tecnologia (FCT), Portugal, through project UID/04561/2025
(DOI: \nolinkurl{10.54499/UID/04561/2025}). Emma D'Aniello acknowledges partial support
from the Erasmus+ programme, the 2024 DYNAMIChE project of the INdAM Group
GNAMPA, PRIN 2022 QNT4GREEN, the DAISY project of Universit\`a degli Studi
della Campania ``Luigi Vanvitelli'' (D.R.~111, 09/02/2024), and the UMI-TAA
Group on Approximation Theory and Applications. Henrique M. Oliveira
acknowledges partial support from FCT, Portugal, through CAMGSD -- Centro de
An\'alise Matem\'atica, Geometria e Sistemas Din\^amicos, under projects
UID/04459/2025 (DOI: \nolinkurl{10.54499/UID/04459/2025}) and UID/PRR/04459/2025
(DOI: \nolinkurl{10.54499/UID/PRR/04459/2025}), financed by national funds and by the
European Union through the Recovery and Resilience Plan (PRR) --
NextGenerationEU.

\paragraph{Competing interests.}
The authors declare that they have no competing financial or non-financial
interests that could have influenced the work reported in this manuscript.

\begingroup
\small
\bibliographystyle{plain}
\bibliography{reciprocal_gradient_theory_refs}
\endgroup
\end{document}